\documentclass[paper=a4, fontsize=11pt]{scrartcl} 

\usepackage[T1]{fontenc} 
\usepackage{fourier} 
\usepackage[english]{babel} 
\usepackage{amsmath,amsfonts,amsthm} 
\usepackage{amsmath, calligra, mathrsfs}
\usepackage{amssymb}
\usepackage[mathscr]{euscript}
\usepackage{verbatim}
\usepackage[dvipsnames]{xcolor}
\usepackage{mdframed} 
\usepackage{soulutf8}
\usepackage{stmaryrd}
\usepackage{tikz-cd}
\usepackage{hyperref}
\usepackage{lipsum} 

\usepackage{sectsty} 
\allsectionsfont{\centering \normalfont\scshape} 

\usepackage{fancyhdr} 
\usepackage{xurl}
\newcommand*{\sheafhom}{\mathcal{H}\kern -.5pt om}
\numberwithin{equation}{section} 
\numberwithin{figure}{section} 
\numberwithin{table}{section} 

\newtheorem{thm}{Theorem}[section]
\newtheorem{cor}[thm]{Corollary}
\newtheorem{prop}[thm]{Proposition}
\newtheorem{obs}[thm]{Observation}

\theoremstyle{definition}
\newtheorem{defn}[thm]{Definition}

\theoremstyle{remark}
\newtheorem{rem}[thm]{Remark}

\DeclareMathOperator{\cone}{cone}

\DeclareMathOperator{\rk}{rank}

\newcommand{\horrule}[1]{\rule{\linewidth}{#1}} 

\title{	
	\normalfont \normalsize 
	\textsc{} \\ [25pt] 
	\horrule{0.5pt} \\[0.4cm] 
	\huge Anti-Ramsey theory problems, lattice point counts on polytopes, and Hodge structures on the cohomology of toric varieties 

	\horrule{2pt} \\[0.5cm] 
}

\author{Soohyun Park} 

\date{\normalsize June 14, 2022} 

\begin{document}
	
	\maketitle 
	
	\begin{abstract}
		\noindent We find families of graphs $G$ and subgraphs $H$ of $G$ such that the number of edge colorings of $G$ avoiding a monochromatic coloring of $H$ is determined by lattice point counts or a Hodge structure on the cohomology of a certain toric variety. In general, this gives a class of ``anti-Ramsey theory problems'' with a geometric structure. For example, we find one for Ramsey numbers of classes of such graphs. The key observation is that our previous result expressing simplicial chromatic polynomials in terms of $h$-vectors of auxiliary simplicial complexes \cite{Pa} can be reinterpreted as one on edge colorings of graphs avoiding monochromatic colorings of specified forbidden subgraphs. Specializing to simplicial complexes arising from triangulations of polytopes (e.g. unimodular triangulations), we obtain families of graphs and forbidden subgraphs where edge colorings avoiding monochromatic colorings of the forbidden subgraphs depend on lattice point counts or Hodge structures on the cohomology of toric varieties. 
	\end{abstract}
	
	\section{Introduction}
	
	Roughly speaking, Ramsey theory studies structures which are forced to contain a ``regular'' substructure when they are large enough. For example, Ramsey numbers (and graph-theoretic instances in general) give a threshold for this size in order for an edge coloring of a complete graph of some size to contain a monochromatic subgraph (e.g. a clique). There has been extensive work on related to these kinds of questions and an overview is given in a survey of Conlon--Fox--Sudakov \cite{CFS}. \\
	
	The specific types of problems which we will focus on have to do with \emph{avoiding} monochromatic structures in edge colorings. We will term these ``anti-Ramsey problems''. An overview of such ``forbidden graph'' problems is given by Bollob\'as in \cite{Bol}. Some examples of results in this direction include those of Alon--Balogh--Keevash--Sudakov \cite{ABKS} on avoiding monochromatic cliques and Yuster \cite{Yus} on avoiding monochromatic triangles. Fujita--Liu--Magnant \cite{FLM} and Kano--Li \cite{KL} give surveys involving results on edge colorings avoiding monochromatic colorings of specified structures. A common property of results of this type (and extremal graph theory problems in general) is that they are often asymptotic or involve some kind of quantitative bound. In this work, we give families of graphs where this problem can be studied from a \emph{structural} perspective coming from topology. This enables us to obtain specializations where the colorings are parametrized by lattice points or generating functions from Hodge structures on the cohomology of toric varieties. \\
	
	Our starting point is a reinterpretation of the simplicial chromatic polynomial (Definition \ref{simpchromdef}, Proposition \ref{simpmonint}) as a count of edge colorings of graphs avoiding monochromatic colorings of subgraphs corresponding to the minimal nonfaces of the given simplicial complex. Note that the simplicial chromatic polynomial of \emph{any} simplicial complex has such an interpretation. The topological perspective comes from previous results which express the simplicial chromatic polynomials of simplicial complexes whose minimal nonfaces satisfy appropriate intersection properties in terms of $h$-vectors of auxiliary simplicial complexes (Theorem \ref{simpchromsrhilb}, Corollary \ref{simpchromhilbpol}). This gives a family of graphs $G$ and subgraphs $\{ H_i \}$ where edge colorings avoiding monochromatic colorings of $H_i$ are parametrized by $h$-vectors of simplicial complexes. A consequence is an interpretation of Ramsey numbers of classes of graphs in terms of the topology of certain configuration spaces (Corollary \ref{ramconseq}) Specializing to instances where the auxiliary simplicial complexes arise from unimodular triangulations of polytopes, we find cases where such colorings are parametrized by lattice point counts of (dilations of) polytopes (Theorem \ref{colorlatticecoh}). Finally, a specialization to compressed polytopes (Theorem \ref{simpchromcohsum}) gives cases where they are parametrized by (truncated) generating functions of Hodge structures on the cohomology of toric varieties. \\
	
	The expression of simplicial chromatic polynomials in terms of $h$-vectors of auxiliary simplicial complexes (Theorem \ref{simpchromsrhilb}, Corollary \ref{simpchromhilbpol}) and their reinterpretation in terms of edge colorings of graphs (Proposition \ref{simpmonint}) is given in Section \ref{simpramcon}. Afterwards, we specialize to instances where the auxiliary simplicial complexes arise from polytopes with unimodular triangulations to obtain find subfamilies of graphs with colorings considered coming from lattice point counts of polytopes (Theorem \ref{colorlatticecoh}) in Section \ref{latcolor}. Finally, we consider edge colorings avoiding monochromatic colorings of specified subgraphs parametrized by truncated generating functions depending on Hodge structures on the cohomology of certain toric hypersurfaces (Theorem \ref{simpchromcohsum}) in Section \ref{torcon}.
	
	\section*{Acknowledgements}
	
	I am very grateful to my advisor Benson Farb for his guidance and encouragement. Also, I would like to thank Jodi McWhirter, Vic Reiner, and Jesse Selover for helpful discussions.
	
	\section{Connections between simplicial chromatic polynomials and anti-Ramsey-type problems}  \label{simpramcon}
	
	Simplicial chromatic polynomials were originally introduced by Cooper--de Silva--Sazdanovic \cite{CdSS} as a ``categorification'' of the observation that Euler characteristics of ordered configuration spaces of points can often be parametrized using chromatic polynomials. When the minimal nonfaces of the simplicial complex satsify appropriate intersection properties, we studied this polynomial from a combinatorial point of view \cite{Pa} and expressed simplicial chromatic polynomials in terms of Hilbert series of Stanley--Reisner rings \cite{Pa}. The notation we will use writes $S$ for the simplicial complex and $V$ for the vertex set of the simplicial complex. Before discussing new material, we will first review the old result and the definitions used there.  \\
	
	Here is the definition of a simplicial chromatic polynomial (from \cite{Pa} in reference to \cite{CdSS}):
	
	\begin{defn}  (Definition 2.1 on p. 725 and p. 738 of \cite{CdSS}) \\
		Let $S$ be a simplicial complex whose $0$-skeleton is given by the vertex set $V = V(S) = \{ v_1, \ldots, v_n \}$. Let $M$ be a topological space. For each simplex $\sigma = [ v_{i_1} \cdots v_{i_k} ]$, define the diagonal corresponding to $\sigma$ to be \[ D_\sigma = \{ (x_1, \ldots, x_n) \in M^n : x_{i_1} = \cdots  = x_{i_k} \}. \]
		
		We define the \textbf{simplicial configuration space} as 
		
		\begin{equation}
			M_S = M^n \setminus \bigcup_{\sigma \in \Delta^V \setminus S} D_\sigma \label{simpconfdefuni}
		\end{equation}
		
		where $\Delta^V$ is the simplicial complex containing all subsets of the vertices $v_i$ (analogous to a simplex generated by independent vectors corresponding to the $v_i$) and $\Delta^V \setminus S$ denotes tuples of vertices in $V$ which do \emph{not} occur as simplices in $S$. 
	\end{defn}
	
	\begin{defn} \label{simpchromdef} (Definition 6.1 on p. 738 of \cite{CdSS}) \\
		Let $S$ be a simplicial complex and let $M$ be a manifold. Given $S$ and $M$, let \[ \chi_c(S, M) := \sum (-1)^k \rk H_c^k(M_S). \]
		
		The \textbf{simplicial chromatic polynomial} of a simplicial complex $S$ is the polynomial defined by the assignment $\chi_c(S) : t \mapsto \chi_c(S, \mathbb{CP}^{t - 1})$.
	\end{defn}
	
	For simplicial complexes $S$ whose minimal nonfaces satisfy appropriate intersection properties, we can show that their simplicial chromatic polynomials are determined by $h$-vectors of auxiliary simplicial complexes $T(S)$.
	
	\begin{defn} \label{propidef}
		Let $S$ be a simplicial complex with minimal nonfaces $\sigma_1, \ldots, \sigma_r$. A simplicial complex $S$ satisfies \textbf{property $\mathbf{I}$} if there is a collection of finite sets $\alpha_i$ such that $|\alpha_i| = |\sigma_i| - 1$ for each $1 \le i \le r$ and $\alpha_I \cap \alpha_p = \sigma_I \cap \sigma_p = \emptyset$ if $\sigma_I \cap \sigma_p = \emptyset$ and $|\alpha_I \cap \alpha_p| = |\sigma_I \cap \sigma_p| - 1$ if $|I| \ge 2$ and $\sigma_I \cap \sigma_p \ne \emptyset$ for each subset $I \subset [r]$ and $p \notin I$.
	\end{defn}

	\begin{rem} \label{anypropi}
		Note that \emph{any} simplicial complex can be set equal to $T(S)$ for some simplicial complex $S$ satisfying property $I$. For example, add a new extraneous vertex to every minimal nonface of $T(S)$.
	\end{rem}
	
	\begin{defn} \label{compdef}  
		Given a $k$-element subset $I = \{ \sigma_1, \ldots, \sigma_k \} \subset \Delta^V \setminus S$, let $G_I$ be the graph whose vertices are the $\sigma_i$ with two vertices corresponding to $\sigma_i, \sigma_j$ being connected by an edge if and only if $\sigma_i \cap \sigma_j \ne \emptyset$. Let $c(I)$ be the number of connected components of $G_I$. \\
	\end{defn}
	
	\begin{thm} \label{simpchromsrhilb} (Theorem 1.5 on p. 4 and Proposition 2.8 on p. 9 of \cite{Pa}) ~\\
		\vspace{-3mm}
		\begin{enumerate}
			\item  Let $r$ be the number of minimal nonfaces of $S$. If $c(I) = 1$ for all $I \subset [r]$, then 
			
			\begin{align*}
				\chi_c(S)(t) - t^n &= t^{n + 1} ((1 - t^{-1})^{n - d} h_S(t^{-1}) - 1 ) \\
				&= t^{d + 1}( (t - 1)^{n - d} h_S(t^{-1}) - t^{n - d} ) \\
				\Longrightarrow \frac{\chi_c(S)(t) - t^n + t^{n + 1} }{t(t - 1)^{n - d}} &= t^d h_S(t^{-1}).  
			\end{align*}
			
			\item Suppose that $S$ satisfies property $I$. Then, there is some auxiliary simplicial complex $T(S)$ such that 
			
			\begin{align*}
				\chi_c(S)(t) &= t^n (1 - t^{-1})^{n - d} h_{T(S)}(t^{-1}) \\
				&= t^d (t - 1)^{n - d} h_{T(S)}(t^{-1}) \\
				\Longrightarrow \frac{\chi_c(S)(t)}{t^{d - m} (t - 1)^{n - d}} &= t^m h_{T(S)}(t^{-1})
			\end{align*}
			
			formally, where $m = \dim T(S)$. 
			
		\end{enumerate}
	\end{thm}

	There are also other intersection patterns of the minimal nonfaces of $S$ where the simplicial chromatic polynomial is determined by the $h$-vector of an auxiliary simplicial complex $T(S)$.

	\begin{cor} \label{simpchromhilbpol} (Corollary 1.8 on p. 5 of \cite{Pa}) \\
		Let $S$ be a simplicial complex and $\sigma_1, \ldots, \sigma_r$ be the minimal nonfaces of $S$. Suppose that $c(I) = a$ for all subsets $I \subset [r]$ with $|I| \ge 2$. If $h_{a + r} \ge 1$, $h_{a + 1}, h_{a + 2} \ge 3$, and $h_i \ge 1$ for all $i \ge a$, then \[ t^{-n} - \chi_c(S)(t^{-1}) + (t^{n + 1} - t^{n + a}) \sum_{\sigma_i} t^{-|\sigma_i|} = t^{-n} P(t^{-1}) \] for the Hilbert polynomial $P = P(x)$ of some $k$-algebra. \\
		
		If $a = 1$, this specializes to \[ t^{-n} - \chi_c(S)(t^{-1}) = t^{-n} P(t^{-1}). \]
	\end{cor}

	As noted in Proposition 2.2 on p. 726 of \cite{CdSS}, the simplicial chromatic polynomial specializes to the ``usual'' chromatic polynomial of a graph when the simplicial complex $S$ is the independence complex $I(G)$ of some graph $G$. However, we can consider the terms involved from a different perspective if $S$ is taken to parametrize ``monotone'' properties of graphs (i.e. $P$ such that $P$ true for $G \Longrightarrow P$ true for $H$ for any subgraph $H$ of $G$). Other such examples are on p. 99 -- 100 of \cite{J}.
	
	\begin{prop} \label{simpmonint}
		Let $[n] = \{ 1, \ldots, n \}$ and $G$ be a graph whose vertices are labeled by $[n]$. Consider a simplicial complex $S(G)$ whose vertices are labeled by the edges of $G$. Then, $\chi_c(S)(t)$ gives the number of edge colorings of $G$ using $\le t$ colors that avoid monochromatic colorings of collections of edges parametrized by the minimal nonfaces of $S(G)$. In some sense, the minimal nonfaces of $S(G)$ correspond to ``minimal forbidden subgraphs''. \\
		
		Note that \emph{any} simplicial complex can be written as $S(G)$ for some graph $G$. Also, any property $P$ of a graph preserved by its subgraphs can be parametrized by a simplicial complex (which we denote by $S(P)$).
	\end{prop}

	\begin{proof}
		The first part follows from the definition of $\chi_c(S)$ since $\sigma \subset \tau \Longrightarrow D_\sigma \supset D_\tau$, which means that it suffices to consider minimal nonfaces. The omitted subsets $D_\sigma$ parametrize colorings of the edges of $G$ where the edges corresponding to elements of $\sigma$ all have the same color. The second statement follows from since we can take the smallest ``forbidden'' graphs to be minimal nonfaces.  	
	\end{proof}
	
	\begin{obs} \label{extramtopstart} \normalfont 
		\vspace{3mm}
		Proposition \ref{simpmonint} indicates that the simplicial chromatic polynomial $\chi_c(S)(t)$ counts edge colorings which are the ``complement'' of what is studied in graph Ramsey theory. In general, combining Proposition \ref{simpmonint} with Theorem \ref{simpchromsrhilb} and Corollary \ref{simpchromhilbpol} gives a family of extremal graph theory problems (i.e. edge colorings avoiding monochromatic ``forbidden subgraphs'') which are determined by $h$-vectors of simplicial complexes (equivalently by $f$-vectors of simplicial complexes). This gives a topological point of view on extremal graph problems, which have results that are mainly stated in terms of inequalities or focus on specific types of subgraphs where we want to avoid monochromatic colorings. Some examples include Bollob\'as' overview in \cite{Bol}, results of Alon--Balogh--Keevash--Sudakov \cite{ABKS} on edge colorings avoiding monochromatic cliques, results on Yuster \cite{Yus} on edges avoiding monochromatic triangles, and surveys of Kano--Li \cite{KL} and Fujita--Liu--Magnant \cite{FLM} involving monochromatic structures in edge colorings. We also provide some methods of studying monochromatic structures in edge colorings of graphs which are \emph{not} complete graphs, which do not appear to be studied as frequently in Ramsey theory-related literature. \\
	\end{obs}

	For example, the interpretation of $\chi_c(S)(t)$ given in Proposition \ref{simpmonint} can be connected to both the ``usual'' Ramsey numbers $r(H)$ for a graph with vertices labeled by $[n] = \{ 1, \ldots, n \}$ (p. 49 of \cite{CFS}) and Ramsey numbers of \emph{classes of graphs}, which are analogues of Ramsey numbers restricted to some collection of graphs. The following definition is an extension of that used by Belmonte--Heggernes--van't Hof--Rafiey--Saei in \cite{BHvRS}.

	\begin{defn} \label{classramnum} (Belmonte--Heggernes--van't Hof--Rafiey--Saei in \cite{BHvRS}) \\
		Fix a positive integer $t$. Given a graph class $\mathcal{G}$ (i.e. some finite collection of graphs), the Ramsey number $R_{\mathcal{G}}(i, j)$ is the smallest number such that if the edges of a graph in $\mathcal{G}$ is colored with $2$ different colors (say red and blue), then it contains a monochromatic blue clique on $i$ vertices or a red clique on $j$ vertices. 
	\end{defn}

	\begin{cor} \label{ramconseq} ~\\
		\vspace{-5mm}
		
		Setting $S = S(G)$ from Proposition \ref{simpmonint}, we find that the simplicial chromatic polynomial has a natural relationship with Ramsey numbers and its generalizations studied in the literature. 
		\begin{enumerate}
			\item $G = K_n$ case 
				\begin{enumerate}
					\item The simplicial chromatic polynomial $\chi_c(S(K_n))(t)$ gives the number of edge colorings of $G$ using $\le t$ colors which do \emph{not} have any monochromatic cliques of size $i$. Substituting $t = 2$ into the simplicial chromatic polynomial $\chi_c(S(K_n))(t)$, we find that the Ramsey number $R(i, i)$ is given by the minimal $n$ such that $\chi_c(S(K_n))(2) = 0$. 
					
					\item More generally, consider a graph $H$ on the vertex set $[n]$ in the setting of part $2$. The Ramsey number $r(H)$ is the smallest number $n$ such that any coloring of the edges $K_n$ using $2$ colors contains a monochromatic coloring of the edges of $H$. Given a graph $G$ and a subgraph $H$ of $G$, let $S_H(G)$ be the simplicial complex with vertices given by the edges of $G$ and minimal nonfaces given by edges coming from copies of $H$ in $G$. Then, we have that $r(H)$ is the smallest number $n$ such that $\chi_c(S_H(K_n))(2) = 0$. 
					
				\end{enumerate}
			\item Other graphs $G$
				\begin{enumerate}
					\item For each graph $G$ in a graph class $\mathcal{G}$, let $S(G)$ be the simplicial complex described in Proposition \ref{simpmonint} with the minimal nonfaces given by edges of cliques on $i$ vertices contained in $G$. Substituting $t = 2$ into the polynomials $\chi_c(S(G))(t)$ for $G \in \mathcal{G}$, we have that $R_{\mathcal{G}}(i, i)$ is the smallest number $N$ such that $\chi_c(S(G))(2) = 0$ for all $G \in \mathcal{G}$. 
					
					\item In the setting of Proposition \ref{simpmonint}, take $S = S(G)$ with the forbidden graphs given by cycles of length $\ell$. Let $n$ be the number of vertices in a graph $G$ and $\ell$ be a positive integer such that $4 \le \ell \le \frac{n}{8}$. If the minimum degree of among the vertices the graph is $\ge \frac{3n}{4}$, then $\chi_c(S)(2) = 0$. 
					
					\item As in Part 2, consider the simplicial complex $S = S(G)$ from Proposition \ref{simpmonint}. Let $N$ be the number of vertices of $G$ and $M$ be the number of edges of $G$. Suppose that $M \ge N$ and fix a positive integer $t$ such that $M \ge tN$. 
					
						\begin{itemize}
							\item If the minimal nonfaces of $S(G)$ are given by paths of length $\ge \left\lceil \frac{2M}{tN} \right\rceil$, then $\chi_c(S(G))(t) = 0$. 
							
							\item If the minimal nonfaces of $S(G)$ are given by cycles of length $\ge \left\lceil \frac{2M}{t(N - 1)} \right\rceil$, then $\chi_c(S(G))(t) = 0$.
						\end{itemize}

				\end{enumerate}

		\end{enumerate}
		
	\end{cor}
	
	\begin{proof}
		\begin{enumerate}
			\item Parts a and b are applications of Proposition \ref{simpmonint} with the minimal nonfaces taken to be the edges corresponding to $i$-cycles and copies of $H$ contained in $G$ respectively.
			
			\item
				\begin{enumerate}
					\item This is an application of Definition \ref{classramnum} to Proposition \ref{simpmonint}
					
					\item This is an application of Theorem 2.2.8 on p. 13 and Theorem 2.2.9 on p. 14 of \cite{FLM} to Proposition \ref{simpmonint}.
					
					\item This is an application of Theorem 40 on p. 249 of \cite{KL} to Proposition \ref{simpmonint}.
				\end{enumerate}
		\end{enumerate}
	\end{proof}

	\section{Connection with lattice point counts of polytopes} \label{latcolor}

	In this section, we expand the connection between Ramsey-type problems and topological/geometric properties summarized in Observation \ref{extramtopstart} in Section \ref{simpramcon}. Recall that this came from an interpretation (Proposition \ref{simpmonint}) of the simplicial chromatic polynomial $\chi_c(S)(t)$ in terms of edge colorings avoiding monochromatic colorings of certain ``forbidden subgraphs'' (e.g. cycles, cliques, or paths of a certain size) and its expression in terms of $h$-vectors of auxiliary simplicial complexes $T(S)$ when the minimal nonfaces of $S$ satisfy certain intersection properties (Theorem \ref{simpchromsrhilb} and Corollary \ref{simpchromhilbpol}). Note that \emph{any} simplicial complex can be set to be $T(S)$ for \emph{some} simplicial complex $S$ satisfying property $I$. 
	
	\subsection{Background on lattice point counts and heuristics}

	Before we start stating specific expressions/identities, we go over definitions of the objects used. \\

	\begin{defn} (p. 275 of Bruns--Herzog)
		Let $P \subset \mathbb{R}^N$ be a convex bounded polytope of dimension $r$. The \textbf{Ehrhart function} is defined as \[  E(P, m) =  \left| \left\{ z \in \mathbb{Z}^n : \frac{z}{m} \in P \right\}  \right| = | \{ z \in \mathbb{Z}^n : z \in mP \} |  \]
		
		for $m \in \mathbb{N}$ and $m > 0$. Note that $E(P, 0) = 1$. \\
		
		Its generating function is the \textbf{Ehrhart series} \[ E_P(t) := \sum_{m \in \mathbb{N}} E(P, m) t^m. \]
	\end{defn}
	
	The expression of the Ehrhart series as a Hilbert series comes from a more general framework using Hilbert functions to describe generating functions associated to combinatorial objects arising as solutions to homogeneous linear Diophantine equations in $n$ variables (p. 274 -- 276 of \cite{BH}). By Lemma 4.1.4 on p. 149 of \cite{BH} we can use to write \[ E_P(t) = \frac{h^*_0 + h^*_1 t + \ldots + h^*_r t^r}{(1 - t)^{r + 1}} \text{ (p. 3 of \cite{Mus}).} \]  This can either done by using the fact that  $E_P(t)$ is the Hilbert series of a $k$-algebra of dimension $r + 1$ (p. 276 of \cite{BH}) or a direct computation using the fact that $E(P, m)$ takes integer values for every $m \in \mathbb{Z}$ (Lemma 4.1.4 on p. 149 of \cite{BH}, p. 3 of \cite{Mus}). The vector $h^* = (h^*_0, \ldots, h^*_r)$ is called the \textbf{$h^*$-vector} of the polytope $P$. Let $h^*_P(t) = h^*_0 + h^*_1 t + \ldots + h^*_r t^r = (1 - t)^{r + 1} E_P(t)$.  \\

	We can use this expression of $E_P(t)$ as a Hilbert series to study formal analogues of the simplicial chromatic polynomial. The families of simplicial complexes that we have considered make use of the Hilbert series of the Stanley--Reisner ring with $t^{-1}$ substituted in place of $t$. In the case of Ehrhart series, there is a natural symmetry between such polynomials since $(t - 1)^{r + 1} E_P(t^{-1}) = t^{r + 1} h_P(t^{-1})$. Making the corresponding substitutions, Part 1 of Theorem \ref{simpchromsrhilb} implies that 
	\begin{equation} \label{formal1}
		(t - 1)^{n + 1} E_P(t^{-1}) \text{ ``$=$'' } \chi_c(S)(t) - t^n + t^{n + 1} 
	\end{equation}   
	formally, where $r= \dim P$ (used in place of $d = \dim S$) and $n = |V|$ (size of the vertex set). Similarly, applying Part 2 of Theorem \ref{simpchromsrhilb} implies that
	
	\begin{equation} \label{formal2} 
		(t - 1)^{n + 1} E_P(t^{-1}) \text{ ``$=$'' } \frac{\chi_c(S)(t)}{t^{d - r}},
	\end{equation}
	
	where $d = \dim S$, $r = \dim P$ and $n = |V|$ as above. In this case, we use $r$ in place of $m = \dim T(S)$. Note that $\widetilde{E}_P(t) = -E_P(t^{-1})$ in the expressions above, where $\widetilde{E}_P(t) = \sum_{m \ge 1} E(P, -m) t^m$. By a reciprocity result of Ehrhart ((0.3) on p. 166 of \cite{Hi}, Theorem 6.3.11 on p. 276 of \cite{BH}), we have that $(-1)^r E(P, -m) = \#(m(P - \partial P) \cap \mathbb{Z}^N)$ for every integer $m > 0$. This implies that the simplicial chromatic polynomial is formally a normalization of the generating function for lattice points of integer factor dilations of the polytope $P$ with its boundary removed.  \\

	\subsection{Edge colorings avoiding forbidden subgraphs and lattice point counts} \label{latcon}

	In this subsection, we consider some instances where the $h$-vectors of the simplicial complexes such as those considered in Part 1 of Theorem \ref{simpchromsrhilb} are actually \emph{equal} to the $h$-vectors associated to convex polytopes so that this gives an equality. As a consequence, we find that lattice point counts of certain polytopes are determined by the number of ways to color the vertex set of some simplicial complex (within $\le t$ colors for some $t$ when considered as a polynomial in $t$) so that no two vertices lying in the same minimal nonface have the same color (Corollary \ref{colorlatticecoh}). The latter follows from the definition of the simplicial chromatic polynomial. \\

	We will now cover some known cases where the $h$-vector associated to a polytope is equal to that of an actual simplicial complexes.

	We first recall a result giving an equality between $\delta$-vectors of polytopes and $h$-vectors of approrpriate simplicial complexes. More specifically, there is also a result which give an equality between $h$-vectors of polytopes and $h$-vectors of certain simplicial complexes if there is a unimodular triangulation.

	\begin{defn} (p. 694 of \cite{Bra}) \\
		A lattice polytope $P$ satisfies the \textbf{integer decomposition property (IDP)} if $\text{span}_{\mathbb{Z}_{\ge 0}} \{ (1, P) \cap \mathbb{Z}^{n + 1} \} = \cone(P) \cap \mathbb{Z}^{n + 1}$. 
	\end{defn}
	
	\begin{thm} (Bruns and R\"omer, Theorem 1 on p. 67 of  Theorem 4 on p. 698 of \cite{Bra}) \label{hstarhunimod} \\
		If $P$ is Gorenstein and IDP, then $h^*_P$ is the $h^*$-vector of an IDP reflexive polytope. Further, if $P$ admits a regular unimodular triangulation, then there exists a simplicial polytope $Q$ such that $h_P^*$ is the $h$-vector of $Q$, and hence $h_P^*$ is unimodal as a consequence of the $g$-theorem. 
	\end{thm}
	
	\begin{rem}
		Although it is possible for $cP$ to admit a unimodular triangulation while $(c + 1)P$ does not, every sufficiently large dilation of an integral polytope admits a unimodular triangulation by a recent result of Liu (Theorem 1.2 on p. 2 of \cite{Liu}). This builds on an older result of Kempf--Knudsen--Mumford--Saint-Donat in \cite{KKMS}. Also, note that the IDP property is preserved under dilation by definition. Results on related combinatorial invariants are in work of Cox--Haase--Hibi--Higashitani \cite{CHHH}.
	\end{rem}

	Putting this together with the previous results on simplicial polytopes, these imply the following:
	
	\begin{thm} \label{colorlatticecoh} ~\\
		\vspace{-3mm}
		\begin{enumerate}
			\item Suppose that $P$ is a Gorenstein integer polytope admitting a regular unimodular triangulation. Then, the boundary complex of $P$ is abstractly isomorphic to a simplicial complex $T$ such that \[ \frac{\chi_c(S)(t)}{t^{e - u - 1} (t - 1)^{n - e + r + 1}} = (-1)^{r + 1}  E_P^+(t)  \]
			
			for any $S$ is a simplicial complex satisfying property $I$ (Theorem \ref{simpchromsrhilb}) such that $T = T(S)$, where $e = \dim S$, $r = \dim P$, $u = r - u+ 1$, and $E^+(P, m) = \left| \left\{ z \in \mathbb{Z}^n : \frac{z}{m} \in P \setminus \partial P  \right\} \right| = \left| \left\{ z \in \mathbb{Z}^n : z \in m(P \setminus \partial P)  \right\}  \right| $ (p. 275 of \cite{BH}). . \\
			
			The definition of $\chi_c(S)$ then implies that the number of colorings of the vertices of $S$ using at most $t$ colors such that no two vertices of $S$ lie in the same minimal nonface of $S$ divided by $t^{e - m - 1} (t - 1)^{n - e + m + 1}$ is equal to the generating series of lattice point counts on integer dilations of $P \setminus \partial P$. Also, the coefficients of these lattice point counts/colorings are determined by the (primitive) cohomology of hypersurfaces on algebraic tori. \\
			
			\item The class of polytopes in Part 1 induces a class of graphs $G$ and specified subgraphs $\{ H_i \}$ where the edge colorings of $G$ using $\le t$ colors where the $H_i$ are \emph{not} monochromatic are parametrized by $t^a (t - 1)^b$ multiplied by the Ehrhart function of a polytope. This implies that the colorings in question are essentially parametrized by lattice point counts of integer dilations of some polytope.
		\end{enumerate}
	\end{thm}

	\begin{proof}
		\begin{enumerate}
			\item By Theorem \ref{simpchromsrhilb}, the initial conditions imply that $E_P(t) = \frac{h_T(t)}{(1 - t)^{u + 1}}$. Since \emph{any} simplicial complex is equal to $T(S)$ for some simplicial complex satisfying property $I$ (e.g. by adding a single particular new vertex to each minimal nonface), we can set $T = T(S)$ for some simplicail complex $S$ satisfying property $I$ from Theorem \ref{simpchromsrhilb}. Let $e = \dim S$, $n$ be the number of vertices of $S$, and $m = \dim T(S)$. Part 2 of Theorem \ref{simpchromsrhilb} implies that 
			
			\begin{align*}
				\frac{\chi_c(S)(t)}{t^e (t - 1)^{n - e}} &= h_{T(S)}(t^{-1}) = (1 - t^{-1})^{r + 1} E_P(t^{-1}) \\
				&= t^{-r - 1} (t - 1)^{r + 1} E_P(t^{-1}) \\
				\Longrightarrow \frac{\chi_c(S)(t)}{t^{e - r - 1} (t - 1)^{n - e + r + 1} } &= E_P(t^{-1}) \\
				&= (-1)^{r + 1} E_P^+(t),
			\end{align*}
			
			where $E^+(P, m) = \left| \left\{ z \in \mathbb{Z}^n : \frac{z}{m} \in P \setminus \partial P  \right\} \right| = \left| \left\{ z \in \mathbb{Z}^n : z \in m(P \setminus \partial P)  \right\}  \right| $ (p. 275 of \cite{BH}). 
			
			The last equality follows from the Ehrhart reciprocity relation $E_P(t^{-1}) = (-1)^{r + 1} E_P^+(t)$, which is known to be a generating function for lattice points on positive integer dilations of $P \setminus \partial P$ ((0.3) on p. 166 of \cite{Hi}, Theorem 6.3.11 on p. 276 of \cite{BH}).  This also implies the connection between coloring and lattice point counts in the last part of the statement. Finally, the statement connecting coefficients to the (primitive) cohomology of hypersurfaces on algebraic tori follows from replacing $t$ by $t^{-1}$ and applying Corollary \ref{simptorhypcoh}. 
			
			\item This is a combination of Part 1 and Proposition \ref{simpmonint}. Note that \emph{any} simplicial complex an be written as $S(G)$ for some graph $G$ (Remark \ref{anypropi}) and a collection of forbidden subgraphs corresponding to the minimal nonfaces of $S$ where monochromatic edge colorings are not allowed.
		\end{enumerate}
	\end{proof}

	\section{A Hodge structure on the cohomology of toric varieties} \label{torcon}

	Using a similar analysis to the one in Section \ref{latcon}, we give classes of graphs whose Ramsey numbers are determined by the mixed Hodge structure on certain toric varieties. 
	
	\subsection{Background on $\delta$-invariants}

	The expressions above expressing the simplicial chromatic polynomial $\chi_c(S)$ ``formally'' in terms of $E_P(t^{-1})$ are very close to a natural existing invariant of convex polytopes (the $\delta$-vector, p. 166 of \cite{Hi}). Given a convex polytope $P \subset \mathbb{R}^N$ of dimension $N$, its $\delta$-vectors can be expressed in terms of Hodge--Deligne numbers form the primitive part of the middle cohomology of hypersurfaces in algebraic tori $Z(f) = (f = 0) \subset (\mathbb{C}^*)^N$ for $f$ such that $N(f) = P$ (i.e. with Newton polytope $P$). In some sense, this implies that a formalization of the simplicial chromatic polynomial is determined by a portion of the ``primitive'' Hodge--Deligne polynomial of a hypersurface in an algebraic torus. Note that \emph{any} simplicial complex can be set as $T(S)$ in Part 2 of Theorem \ref{simpchromsrhilb} for \emph{some} simplicial complex $S$. In addition, the $h$-vectors associated to convex polytopes are associated to $h$-vectors of actual simplicial complexes under suitable conditions. This turns the formalization into an equality which is realized by an actual simplicial chromatic polynomial.

	\begin{defn}(p. 166 of \cite{Hi}) \\
		Let $P \subset \mathbb{R}^N$ be an integral convex polytope (i.e. with vertices given by integer coordinates), $r = \dim P$, and $\partial P$ be the boundary of $P$. We define the sequence of integers $\delta_0, \delta_1, \delta_2, \ldots$ by the formula 
		
		\begin{equation} \label{deltadef}
			(1 - t)^{r + 1} \left( 1 - \sum_{m = 1}^\infty E(P, m) t^m \right) = (1 - t)^{r + 1} \left( 1 - E_P(t) \right) = \sum_{i = 0}^\infty \delta_i t^i.
		\end{equation}
		
		By a fundamental result on generating functions (equivalence between i and iii in Corollary 4.3.1 on p. 543 of \cite{StaEC1}), we have that $\delta_i = 0$ for every $i > r$. \\
		
		When $P \subset \mathbb{R}^N$ is an integral convex polytope of dimension $r$, we say that the sequence $\delta(P) = (\delta_0, \delta_1, \ldots, \delta_r)$ is the \textbf{$\delta$-vector} of $P$. In particular, we have that $\delta_0 = 1$ and $\delta_1 = \#(P \cap \mathbb{Z}^N) - (r + 1)$. 
	\end{defn}

	There is a subclass of polytopes $P$ whose $\delta$-vectors are equal to $h$-vectors of the simplicial complex corresponding to some triangulation of the boundary $\partial P$. 
	
	\begin{defn} (p. 168 -- 169 of \cite{Hi}) \label{polydefs}  ~\\
		\vspace{-3mm}
		\begin{enumerate}
			\item A polytope $P$ of dimension $r$ is of \textbf{standard type} if $P \subset \mathbb{R}^r$ and the origin of $\mathbb{R}^r$ is contained in the interior $P \setminus \partial P$ of $P$. For each integer $r > 1$, let $\mathcal{C}_0(r)$ be the set of integral convex polytopes in $\mathbb{R}^r$ of standard type.
			
			\item Given a polytope $P$ of standard type, its \textbf{polar set (or dual polytope)} is defined as \[ P^* := \{  (\alpha_1, \ldots, \alpha_r) \in \mathbb{R}^r : \alpha_1 \beta_1 + \ldots + \alpha_r \beta_r \le 1 \text{ for all } (\beta_1, \ldots, \beta_r) \in P  \}. \] For each $r > 1$, let $\mathcal{C}^*(r)$ be the set of $P \in \mathcal{C}_0(r)$ such that $P^*$ is an integral convex polytope. \\
			
			Note that $P^* \subset \mathbb{R}^r$ is also a convex polytope of standard type and $(P^*)^* = P$. Moreover, if $P$ is rational, then $P^*$ is rational. 
			
			\item A triangulation $S$ of the boundary $\partial P$ of $P \in \mathcal{C}^*(r)$ with vertex set $V = \partial P \cap \mathbb{Z}^r$ is called \textbf{compressed}. 
			
		\end{enumerate} 
	\end{defn}
	
	\begin{prop} (Stanley, Betke--McMullen, Proposition 2.2 on p. 171 of \cite{Hi}) \label{comphilbdel} \\
		Suppose that $S$ is a triangulation of the boundary $\partial P$ of $P \in \mathcal{C}^*(r)$ with vertex set $V = \partial P \cap \mathbb{Z}^r$. Let $h(S) = (h_0, \ldots, h_r)$ be the $h$-vector of $S$ and $\delta(P)$ be the $\delta$-vector of $P$. Then, $\delta(P) \ge h(S)$ (i.e. $\delta_i \ge h_i$ for each $1 \le i \le r$). Moreover, $h(S) = \delta(P)$ if and only if $S$ is compressed.
	\end{prop}
	
	We can combine this with Part 2 of Theorem \ref{simpchromsrhilb} to state the following:
	
	\begin{cor} \label{simpehractcomp}
		Suppose that $S$ is a simplicial complex satsifying property $I$ such that $T(S)$ (Theorem \ref{simpchromsrhilb}) is a compressed triangulation of some polytope $P \in \mathcal{C}^*(m + 1)$. Let $d = \dim S$, $n = |V|$ (the size of the vertex set), and $m = \dim T(S)$. Then, we have that  
		
		\[  \frac{\chi_c(S)(t)}{t^{d - m - 2} (t - 1)^{n - d}} = (t^{m + 2} - 1) (1 + \widetilde{E}_P(t)), \]
		
		where 
		
		\[ \widetilde{E}_P(t) = \sum_{m \ge 1} E(P, -m) t^m. \]
		
		Note that $E(P, -m) = (-1)^{m + 2} \# (m(P \setminus \partial P) \cap \mathbb{Z}^{m + 2})$ and \emph{any} simplicial complex can be set equal to $T(S)$ for some simplicial complex $S$ satsifying property $I$.
		
	\end{cor}
	
	\begin{proof}
		This is a combination of Part 2 of Theorem \ref{simpchromsrhilb}, the definition of the $\delta$-vector form \ref{deltadef}, and Proposition \ref{comphilbdel}. By Theorem \ref{simpchromsrhilb}, we have that 
		
		\begin{equation} \label{simphilb2}
			\frac{\chi_c(S)(t)}{t^d (t - 1)^{n - d}} = h_{T(S)}(t^{-1}).
		\end{equation} 
		
		Since we assumed that $T(S)$ is a \emph{compressed} triangulation, we have that $h_i(T(S)) = \delta_i(P)$ for each $1 \le i \le m$. By the definition of the $\delta$-vector in \ref{deltadef}, we have that \[ (1 - t)^{m + 2} \left( 1 - \sum_{a = 1}^\infty E(P, a) t^a \right) = \sum_{u = 0}^\infty \delta_u t^u = \sum_{u = 0}^\infty h_u t^u = h_{T(S)}(t) \] since $\dim P = \dim T(S) + 1 = m + 1$. Note that the ``infinite'' sum actually terminates since terms of degree $> m$ are equal to $0$. We can rewrite this as 
		
		\begin{equation} \label{ehrhilb}
			(1 - t)^{m + 2} (1 - E_P(t)) = h_{T(S)}(t).
		\end{equation}
		
		Substituting in $t^{-1}$ in place of $t$, we can combine \ref{simphilb2} with \ref{ehrhilb} to find that 
		
		\begin{align*}
			\frac{\chi_c(S)(t)}{t^d (t - 1)^{n - d}} &= (1 - t^{-1})^{m + 2} (1 - E_P(t^{-1})) \\
			&= (1 - t^{-1})^{m + 2} (1 + \widetilde{E}_P(t)) \\
			&= t^{-m - 2}(t^{m + 2} - 1) (1 + \widetilde{E}_P(t))  \\
			\Longrightarrow \frac{\chi_c(S)(t)}{t^{d - m - 2} (t - 1)^{n - d}} &= (t^{m + 1} - 1) (1 + \widetilde{E}_P(t)),
		\end{align*}

		where \[ \widetilde{E}_P(t) = \sum_{m \ge 1} E(P, -m) t^m. \] This follows from the identity $\widetilde{E}_P(t) = -E_p(t^{-1})$ (p. 3 of \cite{Mus}).
		
		As mentioned previously, a reciprocity result of Ehrhart ((0.3) on p. 166 of \cite{Hi}) implies that $\widetilde{E}_p(t)$ is a generating function for integer dilations of $P \setminus \partial P$. \\
		
	\end{proof}

	\subsection{Hodge structures on toric varieties vs. edge colorings avoiding monochromatic forbidden subgraphs}
	
	More specifically, we can analyze the connection between $\delta$-vectors of polytopes and simplicial chromatic polynomials in a more geometric method. More specifically, we make use of the fact that the $\delta$-vectors of a polytope $P$ are determined by Hodge--Deligne components of the primitive part of the middle cohomology of a hypersurface (depending on $P$) in a certain algebraic torus by work of Batyrev (Section 3 of \cite{Bat}). Let $P \subset \mathbb{R}^N$ be a polytope of dimension $r$ and $T = (\mathbb{C}^*)^N$.

	\begin{defn} \label{batsetup} (p. 357 -- 358 of \cite{Bat}) ~\\
		\vspace{-3mm}
		\begin{enumerate}
			\item Given a Laurent polynomial $f \in \mathbb{C}[x_1^\pm, \ldots, x_N^\pm]$, let $Z_f \subset (\mathbb{C}^*)^N$ be the affine hypersurface determined by $f$ and $P(f)$ be the Newton polytope of $f$.  
			
			\item Given a polytope $P$, let $L(P) = \{ f \in \mathbb{C}[x_1^\pm, \ldots, x_N^\pm] : P = N(f) \}$.
			
			\item Let $P$ be a polytope and $f = \sum c_m x^m$ be a Laurent polynomial such that $N(f) = P$ (i.e. $f \in L(P)$) Given a face $P' \subset P$, define \[ f^{P'}(x) = \sum_{m' \in P'} c_{m'} x^{m'}. \]
			
			\item Given a Laurent polynomial $g = g(x) \in \mathbb{C}[x_1^\pm, \ldots, x_N^\pm]$, let $g_i$ for $1 \le i \le N$ be the logarithmic derivatives \[ g_i = x_i \frac{\partial}{\partial x_i} g(x) \] of $g$.

			\item A Laurent polynomial $f \in L(P)$ and the corresponding affine hypersurface $Z_f \subset T$ are called \textbf{$P$-regular} if $P(f) = P$, and for every $\ell$-dimensional edge $P' \subset P$ ($\ell > 0$), the polynomial equaitons \[ f^{P'}(x) = f_1^{P'}(x) = \cdots = f_n^{P'}(x) = 0 \] have no common solutions in $T$.
		\end{enumerate}
	\end{defn}
	
	In order to study the relationship between the mixed Hodge structure of $Z_f$ and the Laurent polynomial $f$, Batyrev defines the following ``primitive'' cohomology group.
	
	\begin{defn} \label{primbat} (Definition 3.13 on p. 361 of \cite{Bat}) \\
		The \textbf{primitive part of the cohomology group $H^{N - 1}(Z_f)$} (denoted $PH^{N - 1}(Z_f)$) is the cokernel of the injective mapping $H^{N - 1}(T) \hookrightarrow H^{N - 1}(Z_f)$.  
	\end{defn}
	
	\begin{thm} \label{hdcohdelta} (Batyrev,p. 359, Remark 2.13 on p. 357,  Corollary 3.12, Corollary 3.14, and Remark 3.15 on p. 361 of \cite{Bat}) \\
		Given a smooth affine algebraic variety $V$ and $0 \le k \le \dim V$, let \[ H^k(V) = \mathcal{F}^0 H^k(V) \supset \mathcal{F}^1 H^k (V) \supset \cdots \supset \mathcal{F}^{k + 1} H^k(V)  \]
		be the Hodge filtration. We will use $h^{p, q}(H^k(V))$ to denote the Hodge--Deligne numbers (which also involves the weight filtration -- see p. 359 of \cite{Bat}). This will be done similarly (modulo quotients) for the primitive cohomology (as defined in Definition \ref{primbat}). \\

		Let $Z_f \subset T$ be a $P$-regular affine hypersurface (Part 5 of Definition \ref{batsetup}). 

		\begin{enumerate}
			\item The dimensions of the quotients of consecutive terms of the Hodge filtration are given by
			
			\[  \dim \mathcal{F}^i H^{N - 1}(Z_f)/\mathcal{F}^{i + 1} H^{N - 1}(Z_f) = \sum_{q \ge 0} h^{i, q}(H^{N - 1}(Z_f)) =  \begin{cases} 
				\delta_{N - i}(P) & \text{ if } i < N - 1 \\
				\delta_1(P) + N & \text{ if } i = N - 1.
			\end{cases}
			\]
			
			\item If $i \le N - 1$, we have that 
			
			\[  \dim \mathcal{F}^i PH^{N - 1}(Z_f)/\mathcal{F}^{i + 1} PH^{N - 1}(Z_f) = \sum_{q \ge 0} h^{i, q}(PH_c^{N - 1}(Z_f)) = \delta_{N - i}(P). \] 
			
		\end{enumerate}
		
		Note that the generating function of the $\delta_i$ give the Hilbert-Poincar\'e series of coordinate ring the toric variety associated to $P$ (Definition 2.4 on p. 355 of \cite{Bat}) by a regular sequence of linear terms in the coordinate ring.
		
	\end{thm}
	
	Using Theorem \ref{hdcohdelta}, we can see that the heuristics from \ref{formal1} and \ref{formal2} can be combined with \ref{deltadef} to find ``formal'' expressions for the simplicial chromatic polynomial involving Hodge--Deligne polynomials from the cohomology of hypersurfaces in algebraic tori corresponding to Part 1 and Part 2 of Theorem \ref{simpchromsrhilb} given by

	\begin{equation}
		\sum_{i = 0}^r \delta_i t^i \text{ ``$=$'' } (1 - t)^{r + 1} \left( 1 - \frac{\chi_c(S)(t^{-1}) - t^{-n} + t^{- n - 1}}{(t^{-1} - 1)^{n + 1}} \right) = (1 - t)^{r + 1} \left( 1 - \frac{t^{n + 1} \chi_c(t^{-1}) - t + 1}{(t - 1)^{n + 1}} \right)  \label{combo1}
	\end{equation}
	
	from \ref{formal1} and \ref{deltadef} and
	
	\begin{equation}
		\sum_{i = 0}^r \delta_i t^i \text{ ``$=$'' } (1 - t)^{r + 1} \left( 1 - \frac{t^{d - r} \chi_c(S)(t^{-1})}{(t^{-1} - 1)^{n + 1}} \right) = (1 - t)^{r + 1} \left( 1 - \frac{t^{n + d - r + 1} \chi_c(S)(t^{-1}) }{(t - 1)^{n + 1}} \right)  \label{combo2} 
	\end{equation}
	
	from \ref{formal2} and \ref{deltadef}, where $r = \dim P$ and $n = \dim S$. Note that we ``truncated'' the initial sums above since $\delta_i = 0$ for $i > r$. \\
	
	Heuristically, the expressions \ref{combo1} and \ref{combo2} indicate that the simplicial chromatic polynomials in Theorem \ref{simpchromsrhilb} are formally determined by mixed Hodge structures of hypersurfaces on algebraic tori (via a truncated Hodge--Deligne polynomial) after we replace the $h$-vectors of $S$ or $T(S)$ by those of a convex polytope. In the setting of Corollary \ref{simpehractcomp}, the simplicial chromatic polynomials themselves are determined by the mixed Hodge structure of these hypersurfaces on algebraic tori. We can state this more explicitly. 
	
	\begin{cor} \label{simptorhypcoh}
		Suppose that $S$ is a simplicial complex satsifying property $I$ such that $T(S)$ (Theorem \ref{simpchromsrhilb}) is a compressed triangulation of some polytope $P$. Let $d = \dim S$, $n = |V|$ (the size of the vertex set), and $m = \dim T(S)$. \\
		
		Let $f \in \mathbb{C}[x_1^\pm, \ldots, x_m^\pm]$ a Laurent polynomial which is $P$-regular (Part 5 of Definition \ref{batsetup}) and $Z_f \subset (\mathbb{C}^*)^m$ be the affine hypersurface determined by $f$. For each $1 \le i \le m$, the coefficient of $t^i$ in $\frac{t^n \chi_c(S)(t^{-1})}{(t - 1)^{n - d}}$ is \[ \dim \mathcal{F}^{m + 1 - i} PH^m(Z_f)/\mathcal{F}^{m + 2 - i} PH^m(Z_f) = \sum_{q \ge 0} h^{m + 1 - i, q}(PH^m(Z_f)) \]
		
		Note that this is equal to $\dim \mathcal{F}^{m + 1 - i} H^m(Z_f)/\mathcal{F}^{m + 2 - i} H^m(Z_f) = \sum_{q \ge 0} h^{m + 1 - i, q}(H^m(Z_f))$ if $i < m$.		
	\end{cor}
	
	\begin{proof}
		We start by substituting in $t^{-1}$ in place of $t$ in \ref{simphilb2} from the proof of Corollary \ref{simpehractcomp}. This gives 
		
		\begin{equation} \label{simphilb2ver2}
			h_{T(S)}(t) = \frac{\chi_c(S)(t^{-1})}{t^{-d} (t^{-1} - 1)^{n - d}} = \frac{t^d \chi_c(S)(t^{-1})}{(t^{-1} - 1)^{n - d}} = \frac{t^n \chi_c(S)(t^{-1})}{(t - 1)^{n - d}}.
		\end{equation}
		
		Under the assumptions on the simplicial complex, we have that $h_{T(S)}(t) = \sum_{u = 0}^\infty \delta_u t^u$ ($\delta_u = 0$ for $u > r$). For each $1 \le i \le m$, Theorem \ref{hdcohdelta} implies that the coefficient of $t^i$ is 
		
		\[ \dim \mathcal{F}^{m + 1 - i} PH^m(Z_f)/\mathcal{F}^{m + 2 - i} PH^m(Z_f) = \sum_{q \ge 0} h^{m + 1 - i, q}(PH^m(Z_f)) \]
		
		since $\dim P = m + 1$. \\
		
		The second expression comes from Part 2 of Theorem \ref{hdcohdelta}.
	\end{proof}
	
	We can summarize the discussion above as follows:
	
	\begin{thm} \label{simpchromcohsum} ~\\
		\vspace{-3mm}
		\begin{enumerate}
			\item Let $P \subset \mathbb{P}^N$ be a convex polytope of dimension $r$ and $Z_f \subset T$ be a $P$-regular affine hypersurface. If we replace the $h$-vector of $T(S)$ in Theorem \ref{simpchromsrhilb} by that of $P$, the expression \[ (1 - t)^{r + 1} \left( 1 - \frac{t^{n + d - r + 1} \chi_c(S)(t^{-1}) }{(t - 1)^{n + 1}} \right) \] is replaced by a truncated generating function for the Hodge numbers of $H^{N - 1}(Z_f)$. By Theorem \ref{hstarhunimod}, this gives an actual equality when $T(S)$ is the unimodular triangulation of some Gorenstein IDP polytope with a unimodular triangulation. Also, note that $\chi_c(S)$ paramtrizes edge colorings avoiding monochromatic colorings of the minimal nonfaces of $S$ by Proposition \ref{simpmonint}.  \\
			
			The coefficients of the reciprocal polynomial of the resulting function are given by the dimensions of quotients of consecutive terms of the Hodge filtration on $H^{N - 1}(Z_f)$. 
			
			\item Suppose that $S$ is a simplicial complex satsifying property $I$ such that $T(S)$ (Theorem \ref{simpchromsrhilb}) is a compressed triangulation of some polytope $P$. Let $d = \dim S$, $n = |V|$ (the size of the vertex set), and $m = \dim T(S)$. \\
			
			Let $f \in \mathbb{C}[x_1^\pm, \ldots, x_m^\pm]$ a Laurent polynomial which is $P$-regular (Part 5 of Definition \ref{batsetup}) and $Z_f \subset (\mathbb{C}^*)^m$ be the affine hypersurface determined by $f$. For each $1 \le i \le m$, the coefficient of $t^i$ in \[ \frac{t^n \chi_c(S)(t^{-1})}{(t - 1)^{n - d}} \] is \[ \dim \mathcal{F}^{m + 1 - i} PH^m(Z_f)/\mathcal{F}^{m + 2 - i} PH^m(Z_f) = \sum_{q \ge 0} h^{m + 1 - i, q}(PH^m(Z_f)) \]
			
			Note that this is equal to $\dim \mathcal{F}^{m + 1 - i} H^m(Z_f)/\mathcal{F}^{m + 2 - i} H^m(Z_f) = \sum_{q \ge 0} h^{m + 1 - i, q}(H^m(Z_f))$ if $i < m$.
			
			\item As a consequence of the addition-deletion relation satisfied by simplicial chromatic polynomials, a similar one holds for a normalization of the Hodge--Deligne polynomials parametrizing mixed Hodge structures of the affine hypersurfaces on algebraic tori from Part 2.

			\item As a consequence of Part 1, the edge colorings of a graph $G$ arising from the boundary complex of a simplicial complex which do \emph{not} give monochromatic colorings of specified subgraphs $H_i$ determine a truncated generating function for the Hodge numbers of $H^{N - 1}(Z_f)$ (a toric hypersurface).
		\end{enumerate}
	\end{thm}
	
	\begin{rem}
		Since we are making use of the $\delta$-invariant, there are additional symmetry properties in the case where an polytope and its dual/polar polytope are both integral (Theorem 35.8 on p. 105 of \cite{Hi2}).
	\end{rem}

	Department of Mathematics, University of Chicago \\
	5734 S. University Ave, Office: E 128 \\ Chicago, IL 60637 \\
	\textcolor{white}{text} \\
	Email address: \href{mailto:shpg@uchicago.edu}{shpg@uchicago.edu}

\end{document}